\documentclass[a4paper,10pt]{article}
\usepackage[latin1]{inputenc}
\usepackage[T1]{fontenc}
\usepackage{amsmath}
\usepackage{amsfonts}
\usepackage{amstext}
\usepackage{amsthm}
\usepackage[all]{xy}
\usepackage{geometry}
\usepackage{srcltx}
\usepackage{graphics}
\usepackage{epsfig}
\usepackage{subfigure}
\usepackage{slashed}
\usepackage{color}
\usepackage{amssymb}
\usepackage{srcltx}
\newtheorem{theorem}{Theorem}[section]
\newtheorem{lemma}[theorem]{Lemma}
\newtheorem{defi}[theorem]{Definition}
\newtheorem{rem}[theorem]{Remark}

\DeclareMathOperator{\im}{im}
\DeclareMathOperator{\ind}{ind}
\DeclareMathOperator{\sind}{s-ind}

\DeclareMathOperator{\sfl}{sf}
\DeclareMathOperator{\diag}{diag}

\DeclareMathOperator{\sgn}{sgn}

\DeclareMathOperator{\Vect}{Vect}

\title{A $K$-theoretical Invariant and Bifurcation for Homoclinics of Hamiltonian Systems}
 
\author{Alessandro Portaluri and Nils Waterstraat}

\begin{document}
\date{}
\maketitle


\begin{abstract}
We revisit a $K$-theoretical invariant that was invented by the first author some years ago for studying multiparameter bifurcation of branches of critical points of functionals. Our main aim is to apply this invariant to investigate bifurcation of homoclinic solutions of families of Hamiltonian systems which are parametrised by tori. 
\end{abstract}

\section{Introduction}
Let $X$ be a compact topological space and $\mathcal{H}:X\times\mathbb{R}\times\mathbb{R}^{2n}\rightarrow\mathbb{R}$ a continuous map such that $\mathcal{H}_\lambda:=\mathcal{H}(\lambda,\cdot,\cdot):\mathbb{R}\times\mathbb{R}^{2n}\rightarrow\mathbb{R}$ is $C^2$ for all $\lambda\in X$ and its derivatives depend continuously on $\lambda\in X$. We consider the family of Hamiltonian systems

\begin{equation}\label{Hamiltoniannonlin}
\left\{
\begin{aligned}
Ju'(t)+\nabla_u \mathcal{H}_\lambda(t,u(t))&=0,\quad t\in\mathbb{R}\\
\lim_{t\rightarrow\pm\infty}u(t)&=0,
\end{aligned}
\right.
\end{equation}
where $\nabla_u$ denotes the gradient with respect to the variable $u\in\mathbb{R}^{2n}$, and  

\begin{align}\label{J}
J=\begin{pmatrix}
0&-I_n\\
I_n&0
\end{pmatrix}
\end{align}
is the standard symplectic matrix. If we assume that $\nabla_u \mathcal{H}_\lambda(t,0)=0$ for all $(\lambda,t)\in X\times\mathbb{R}$, then $u\equiv 0$ is a solution of \eqref{Hamiltoniannonlin} for all $\lambda\in X$. The aim of this article is to study bifurcation from this trivial branch of solutions, i.e., to find parameter values $\lambda^\ast\in X$ such that in each neighbourhood of $(\lambda^\ast,0)\in X\times C^1(\mathbb{R},\mathbb{R}^{2n})$ there is $(\lambda,u)$ where $u\neq 0$ is a solution of \eqref{Hamiltoniannonlin}.\\
In the one-parameter case, integral bifurcation invariants for \eqref{Hamiltoniannonlin} were obtained by Pejsachowicz in \cite{Jacobo} for $X=S^1$ and by the second author in \cite{Homoclinics} for $X=[0,1]$. Both constructions are based on a bifurcation theorem for critical points of one-parameter families of functionals from \cite{SFLPejsachowicz}, which we want to outline briefly. Let $I:=[a,b]$ denote a compact interval, let $E$ be a real separable Hilbert space and let us consider $C^2$-maps $f:I\times E\rightarrow\mathbb{R}$, where we assume that $0\in E$ is a critical point of each functional $f_\lambda:=f(\lambda,\cdot):E\rightarrow\mathbb{R}$. The second derivatives of $f_\lambda$ at $0\in E$ yield bounded selfadjoint operators $L_\lambda:E\rightarrow E$ which we assume to be Fredholm. Atiyah, Patodi and Singer introduced in \cite{AtiyahSinger} an integer valued homotopy invariant for paths of selfadjoint Fredholm operators which is called \textit{spectral flow}. The main theorem of \cite{SFLPejsachowicz} states that if $L_a, L_b$ are invertible and the spectral flow of the path $\{L_\lambda\}_{\lambda\in I}$ does not vanish, then there is a bifurcation of critical points from the trivial branch $I\times\{0\}\subset I\times E$, i.e. in every neighbourhood of $I\times\{0\}$ in $I\times E$ there is an element $(\lambda,u)\in I\times E$ such that $u\neq 0$ is a critical point of $f_\lambda$.\\
Apart from \cite{Jacobo} and \cite{Homoclinics}, the bifurcation theorem \cite{SFLPejsachowicz} has been used several times in recent research, e.g. for periodic Hamiltonian systems in \cite{SFLPejsachowiczII} and \cite{CalcVar}, for geodesics in semi-Riemannian manifolds in \cite{PicPorTau}, \cite{MussoPejsachowicz} and \cite{AleBifIch}, and for elliptic systems of partial differential equations in \cite{AleSmaleIndef} and \cite{systemsSzulkin}. In \cite{Ale} the first author attempted to obtain from \cite{SFLPejsachowicz} a $K$-theoretic invariant for multiparameter bifurcation of critical points of families of functionals 

\begin{align}\label{f}
f:X\times E\rightarrow\mathbb{R},
\end{align}
where $X$ is a compact manifold, by following a topological argument of Fitzpatrick and Pejsachowicz from \cite{FiPejsachowiczII} for bifurcation of branches of solutions for operator equations. Roughly speaking, denoting by $B\subset X$ the set of all bifurcation points of the family $f$, the invariant gives a sufficient condition for $B$ to be of codimension $1$ in $X$. Moreover, an application to multiparameter bifurcation for families of geodesics is discussed, which was done before for one-parameter families in \cite{PicPorTau} and \cite{MussoPejsachowicz}.\\
The aim of this article is twofold: Firstly, we weaken the assumptions and strengthen the result of the first author's bifurcation theorem for families of functionals $f:X\times E\rightarrow\mathbb{R}$ from \cite{Ale}. In the former theorem, $X$ is assumed to be a smooth compact and orientable manifold and $f$ is $C^2$. By using a recent theorem due to Pejsachowicz and the second author \cite{BifJac}, we can deal with the much more natural situation that $f$ in \eqref{f} is merely continuous but each $f_\lambda:=f(\lambda,\cdot):H\rightarrow\mathbb{R}$ is $C^2$. Moreover, we give a more precise description of the topology of $B$ following the second author's recent article \cite{NilsBif}. Beside these improvements, we also fix some minor inaccuracies in the statement and proof of the main theorem in \cite{Ale}, for which we provide a detailed and self-contained construction of our bifurcation invariant. Secondly, we investigate our $K$-theoretic bifurcation invariant for homoclinic solutions of families of Hamiltonian systems \eqref{Hamiltoniannonlin}. In \cite{Ale}, applications to families of geodesics in semi-Riemannian manifolds were invoked, however, as we explain below, the bifurcation invariant vanishes in this case due to topological reasons and so the existence of bifurcation cannot be concluded. In contrast, we show that non-trivial examples can be obtained for families of functionals having as critical points the solutions of systems as \eqref{Hamiltoniannonlin}. Our main result states that for every torus $X=T^k$, $k\geq 2$, there is a family of Hamiltonian systems \eqref{Hamiltoniannonlin} having a non-vanishing bifurcation invariant, and consequently the set $B\subset X$ of bifurcation points is at least of dimension $k-1$.\\
The paper consists of two main parts and an appendix. The first part is devoted to the improvement of the multiparameter bifurcation theorem for critical points of families of functionals \eqref{f} from \cite{Ale}. In particular, we give a detailed construction of the bifurcation invariant for which we recall the Atiyah-Jänich bundle for families of Fredholm operators, its variant for selfadjoint operators and the spectral flow in order to make the discussion self-contained. After having stated and proved the bifurcation theorem, we explain briefly why the bifurcation invariant vanishes for the class of examples that was discussed in \cite{Ale}. In the second part we turn to the family of Hamiltonian systems \eqref{Hamiltoniannonlin}. At first we recall that under suitable assumptions there exists a family of functionals as in \eqref{f} whose critical points are the solutions of \eqref{Hamiltoniannonlin}. Finally, we show that, in strong contrast to the functionals obtained from families of geodesics in semi-Riemannian manifolds in \cite{Ale}, here the bifurcation invariant can indeed be non-trivial. Actually, we show that for every torus $T^k$, $k\in\mathbb{N}$ there exist infinitely many Hamiltonian systems \eqref{Hamiltoniannonlin} for which the bifurcation invariant does not vanish and our abstract bifurcation theorem applies. The paper has three appendices in which we collect well-known facts on $K$-theory and characteristic classes, complexifications of real Hilbert spaces and the covering dimension of topological spaces for the convenience of the reader.\\
We assume throughout that the reader is familiar with the notation that we introduce in the appendices. Moreover, as we have to deal with real and complex Hilbert spaces, we generally denote real Hilbert spaces by $E$ and complex Hilbert spaces by $H$. The symbols $\mathcal{L}(E)$ and  $\mathcal{L}(H)$ stand for the usual Banach spaces of bounded linear operators on $E$ and $H$, respectively.


\section{A $K$-theoretical Bifurcation Invariant for Families of Functionals}

In this section we revise and improve the bifurcation theorem for families of functionals \cite{Ale} by our Theorem \ref{mainbif}. To this aim, we begin by introducing all necessary details in Section \ref{sections:preliminaries} for a rigorous proof in Section \ref{section:functionals}.


\subsection{Topological Preliminaries}\label{sections:preliminaries}

\subsubsection{Index Bundle and Selfadjoint Operators}
Let $H$ be a complex separable Hilbert space of infinite dimension and $X$ a compact topological space. We denote by $\Phi(H)\subset\mathcal{L}(H)$ the space of all Fredholm operators on $H$ with the norm topology and we consider continuous families $T:X\rightarrow\Phi(H)$. Our first aim is to recall the construction of the \textit{index bundle} of $T$, which is a $K$-theory class $\ind(T)\in K(X)$ that was introduced independently by Atiyah and Jänich in the Sixties (cf. e.g. \cite{KTheoryAtiyah}, \cite{Jaenich} and \cite{indbundleIch}). As $X$ is compact, it is readily seen that there is a finite dimensional subspace $V\subset H$ such that

\[\im(T_\lambda)+V=H,\quad\lambda\in X.\]
If $P$ is the orthogonal projection onto the orthogonal complement $V^\perp$ of $V$ in $H$, then we get a family of exact sequences

\[H\xrightarrow{T_\lambda} H\xrightarrow{P} V^\perp\rightarrow 0,\]
whose kernels define a vector bundle $E(T,V)$ over $X$ having as total space

\[\{(\lambda,u)\in X\times H:\, T_\lambda u\in V\}.\]
The $K$-theory class

\[\ind(T):=[E(T,V)]-[\Theta(V)]\in K(X)\]
does not depend on the choice of the finite dimensional space $V$ and is called the \textit{index bundle} of the family $T$. If $X$ is a singleton, then $\ind(T)$ is just the integral Fredholm index of the operator $T$. Correspondingly, the index bundle shares several properties with the Fredholm index of single elements in $\Phi(H)$, e.g.,

\begin{itemize}
	\item $\ind(T)=0\in K(X)$ if $T_\lambda$ is invertible for all $\lambda\in X$,
	\item if $h:[0,1]\times X\rightarrow\Phi(H)$ is a homotopy of Fredholm operators, then
	\[\ind(h_0)=\ind(h_1)\in K(X),\]
	\item if $T,M:X\rightarrow\Phi(H)$ are two families, then the \textit{logarithmic property} holds:
	\[\ind(TM)=\ind(T)+\ind(M)\in K(X).\]
\end{itemize}
Moreover, 

\begin{itemize}
	\item if $Y$ is a compact topological space and $f:Y\rightarrow X$ continuous, then
	
	\[\ind(f^\ast T)=f^\ast\ind(T)\in K(Y),\]
	where $f^\ast T:Y\rightarrow\Phi(H)$ is the family defined by $(f^\ast T)_\lambda=T_{f(\lambda)}$, $\lambda\in X$.
\end{itemize}
A striking result due to Atiyah and Jänich says that the index bundle induces a bijection

\begin{align}\label{indbij}
\ind:[X,\Phi(H)]\rightarrow K(X),
\end{align}
where $[X,\Phi(H)]$ denotes the set of all homotopy classes of maps from $X$ to $\Phi(H)$. In other words, $\Phi(H)$ is a classifying space for the $K$-theory functor.\\
Let us now denote by $\Phi_S(H)$ the subspace of $\Phi(H)$ consisting of selfadjoint operators, and let us recall that a selfadjoint operator is Fredholm if and only if its kernel is of finite dimension and its image is closed. Unfortunately, the index bundle $\ind(T)\in K(X)$ is not a useful concept for families $T:X\rightarrow\Phi_S(H)$. Indeed, if we consider the homotopy $h:[0,1]\times X\rightarrow\Phi(H)$ defined by $h_{(s,\lambda)}u=T_\lambda u+is\,u$, then, as selfadjoint operators have real spectra, the operators $h_{(s,\lambda)}$ are invertible for $s\in (0,1]$. Hence, by the first two properties of the index bundle from above, we get that $\ind(T)=0\in K(X)$. Atiyah and Singer introduced in \cite{AtiyahSinger} an index bundle in odd $K$-theory for families of selfadjoint Fredholm operators that is non-trivial in general. They proved that the space $\Phi_S(H)$ has three connected components

\begin{align*}
\Phi^+_S(H)&=\{T\in\Phi_S(H):\, \sigma_{ess}(T)\subset(0,\infty)\},\\
\Phi^-_S(H)&=\{T\in\Phi_S(H):\,\sigma_{ess}(T)\subset(-\infty,0)\},\\
\Phi^i_S(H)&=\Phi_S(H)\setminus(\Phi^+_S(H)\cup\Phi^-_S(H)),
\end{align*}
where

\begin{align}\label{sigmaess}
\sigma_{ess}(T)=\{\mu\in\mathbb{R}:\, \mu\,I_H-T\not\in\Phi(H)\}.
\end{align}
Their main result says that $\Phi^+_S(H)$ and $\Phi^-_S(H)$ are contractible topological spaces, whereas the map $\alpha:\Phi^i_S(H)\rightarrow \Omega(\Phi(H))$ to the loop space of $\Phi(H)$ defined by

\[\alpha(T)[t]=\begin{cases}
\cos(t)\,I_H+i\sin(t)\,T,&\quad 0\leq t\leq\pi\\
\cos(t)\,I_H+i\sin(t)\,I_H,&\quad \pi\leq t\leq2\pi
\end{cases},\]
is a homotopy equivalence. If we denote as before by square brackets homotopy classes of maps and use that the loop functor and the suspension functor are adjoint to each other, we obtain for every compact topological space $X$ a sequence of bijections

\[[X,\Phi_S(H)]\xrightarrow{\alpha}[X,\Omega(\Phi(H))]\rightarrow[\Sigma X,\Phi(H)]\xrightarrow{\ind} \widetilde{K}(\Sigma X)=K^{-1}(X)\]
which yields for every family $T:X\rightarrow\Phi_S(H)$ an odd $K$-theory class

\[\sind(T)\in K^{-1}(X).\]      
Note that the following facts are immediate consequences of the corresponding properties of the index bundle from above:

\begin{enumerate}
	\item $\sind(T)=0\in K^{-1}(X)$ if $T_\lambda$ is invertible for all $\lambda\in X$,
	\item if $h:[0,1]\times X\rightarrow\Phi_S(H)$ is a homotopy of selfadjoint Fredholm operators, then
	\[\sind(h_0)=\sind(h_1)\in K^{-1}(X),\]
	\item if $Y$ is a compact topological space and $f:Y\rightarrow X$ continuous, then
	
	\[\sind(f^\ast T)=f^\ast\sind(T)\in K^{-1}(Y),\]
	where $(f^\ast T)_\lambda=T_{f(\lambda)}$, $\lambda\in Y$.
\end{enumerate}
Finally, let us mention for later reference the following lemma which is an immediate consequence of the homotopy invariance property.

\begin{lemma}\label{indtriv}
If $T:X\rightarrow\Phi_S(H)$ and $K:X\rightarrow\mathcal{L}(H)$ is a family of compact selfadjoint operators, then 

\[\sind(T+K)=\sind(T)\in K^{-1}(X).\]
In particular, if $T_\lambda$ is invertible for all $\lambda\in X$, then $\ind(T+K)=0$.
\end{lemma}


\subsubsection{Spectral Flow and First Chern Class}\label{section-sfl}
For a selfadjoint Fredholm operator $T_0\in\Phi_S(H)$, there exists $\Lambda>0$ such that $\pm\Lambda$ do not belong to the spectrum 

\[\sigma(T_0)=\{\mu\in\mathbb{R}:\,\mu-T_0\,\text{not bijective}\}\] 
of $T_0$, and $\sigma(T_0)\cap[-\Lambda,\Lambda]$ consists only of isolated eigenvalues of finite multiplicity. We set for $-\Lambda\leq c<d\leq\Lambda$

\[E_{[c,d]}(T_0)=\bigoplus_{\mu\in[c,d]}\ker(\mu\,I_H-T_0),\]
and we note that it is readily seen from the continuity of finite sets of eigenvalues (cf. \cite[\S I.II.4]{GohbergClasses}) that there exists a neighbourhood $N(T_0,\Lambda)\subset\Phi_S(H)$ of $T_0$ such that $\pm\Lambda\notin\sigma(T)$ and $E_{[-\Lambda,\Lambda]}(T)$ has the same finite dimension for all $T\in N(T_0,\Lambda)$.\\
Let us now assume that $L:[a,b]\rightarrow\Phi_S(H)$ is a path of selfadjoint Fredholm operators. There is a subdivision $a=t_0<t_1<\ldots<t_N=b$, operators $L_i\in\Phi_S(H)$ and numbers $\Lambda_i>0$, $i=1,\ldots N$, such that the trace of the restriction of the path $L$ to $[t_{i-1},t_i]$ is inside $N(L_i,\Lambda_i)$. The \textit{spectral flow} of the path $L$ is defined by

\begin{align}\label{sfl}
\sfl(L,[a,b])=\sum^N_{i=1}{\left(\dim E_{[0,\Lambda_i]}(L_{t_i})-\dim E_{[0,\Lambda_i]}(L_{t_{i-1}})\right)}\in\mathbb{Z}.
\end{align}  
Roughly speaking, $\sfl(L,[a,b])$ is the number of negative eigenvalues of $L_a$ that
become positive when the parameter $t$ travels from $a$ to $b$ minus the number of positive eigenvalues of $L_a$ that become negative, i.e., the net number of eigenvalues which cross zero.\\
Let us mention the following properties of the spectral flow:

\begin{enumerate}
\item[(i)] If $L:[a,b]\rightarrow\Phi_S(H)$ is a path and $L_c$ invertible for some $c\in(a,b)$, then

	\begin{align*}
	\sfl(L,[a,b])=\sfl(L,[a,c])+\sfl(L,[c,b]).
	\end{align*}
\item[(ii)] If $L'$ is the path $L:[a,b]\rightarrow\Phi_S(H)$ traversed in opposite direction, then

\begin{align*}
\sfl(L',[a,b])=-\sfl(L,[a,b]).
\end{align*}

\item[(iii)] If $L:[a,b]\rightarrow\Phi_S(H)$ is such that $L_t$ is invertible for all $t\in [a,b]$, then $\sfl(L,[a,b])=0$.

\item[(iv)] Let $h:[0,1]\times [a,b]\rightarrow\Phi_S(H)$ be a continuous map such that $h(s,a)$ and $h(s,b)$ are invertible for all $s\in[0,1]$. Then

\[\sfl(h(0,\cdot),[a,b])=\sfl(h(1,\cdot),[a,b]).\]

\item[(v)] If $L_t\in\Phi^+_S(H)$ for all $t\in[a,b]$, then

\[\mu_{Morse}(L_t):=\dim\left(\bigoplus_{\mu<0}{\ker(\mu\,I_H-L_t)}\right)<\infty,\quad t\in[a,b],\]
then 
\[\sfl(L,[a,b])=\mu_{Morse}(L_a)-\mu_{Morse}(L_b).\]
\end{enumerate}
The spectral flow of a closed path can be computed topologically by using the index bundle and the first Chern number $c_1:K^{-1}(S^1)\rightarrow H^1(S^1;\mathbb{Z})$ as introduced in Appendix \ref{appendix-bundles} (cf. \cite[Lemma 1.13]{BoossDesuspension} and also \cite{BoWoBuch}). Let $L:[a,b]\rightarrow\Phi_S(H)$ be a closed path, i.e. $L_a=L_b$. Then $L$ induces a family of selfadjoint Fredholm operators parametrised by $S^1$, which we denote by $\hat{L}:S^1\rightarrow\Phi_S(H)$. 

\begin{lemma}\label{Booss}
If $L:[a,b]\rightarrow\Phi_S(H)$ is a closed path, then

\[\sfl(L,[a,b])=\langle c_1(\sind \hat{L}),[S^1]\rangle, \]
where $[S^1]$ denotes the standard generator of $H_1(S^1;\mathbb{Z})$.
\end{lemma}
\noindent
Finally, let us mention that the spectral flow can be defined verbatim for paths $L:[a,b]\rightarrow\Phi_S(E)$ of selfadjoint Fredholm operators on a real Hilbert space $E$ by formula \eqref{sfl}. From \eqref{complexeigenspace}, we immediately obtain that

\begin{align}\label{realcomplex}
\sfl(L,[a,b])=\sfl(L^\mathbb{C},[a,b]),
\end{align}
which means that the spectral flow is invariant under complexification.


\subsection{The Bifurcation Theorem}\label{section:functionals}
In this section we let $E$ be a real separable Hilbert space, $X$ a compact topological space and $f:X\times E\rightarrow\mathbb{R}$ a continuous map such that each $f_\lambda:=f(\lambda,\cdot):E\rightarrow\mathbb{R}$ is $C^2$. We obtain two maps

\begin{align}\label{contder}
\begin{split}
X\times E\ni(\lambda,u)&\mapsto d_uf_\lambda\in \mathcal{L}(E,\mathbb{R}),\\
X\times E\ni(\lambda,u)&\mapsto d^2_uf_\lambda\in \mathcal{L}^2(E,\mathbb{R})
\end{split}
\end{align}
which we assume throughout to be continuous. Here $\mathcal{L}(E,\mathbb{R})$ denotes the Banach space of all bounded linear functionals on $E$, and $\mathcal{L}^2(E,\mathbb{R})$ is the Banach space of all bounded bilinear maps.\\
We say that $u\in E$ is a critical point of $f_\lambda$ if $d_uf_\lambda=0$ and in what follows we assume that $d_0f_\lambda=0$ for all $\lambda\in X$.

\begin{defi}\label{def-bif}
An element $\lambda^\ast\in X$ is called a bifurcation point of critical points of $f$, if every neighbourhood of $(\lambda^\ast,0)$ in $X\times E$ contains elements $(\lambda,u)\in X\times E$ such that $u\neq 0$ is a critical point of $f_\lambda$.
\end{defi}
\noindent
In what follows, we denote by $B\subset X$ the set of all bifurcation points of $f$. Note that $B$ is closed, which is an immediate consequence of Definition \ref{def-bif}.\\
The second derivatives of $f_\lambda$, $\lambda\in X$, at the critical point $0$ yield a continuous family of bilinear forms $\{d^2_0f_\lambda\}_{\lambda\in X}$ in $\mathcal{L}^2(E,\mathbb{R})$, which are symmetric as second derivatives of functionals. Consequently, by the Riesz representation theorem, we obtain a continuous family of bounded symmetric operators $L=\{L_\lambda\}_{\lambda\in X}$ by setting

\[(d^2_0f_\lambda)(u,v)=\langle L_\lambda u,v\rangle_E,\quad u,v\in E.\]
Our final standing assumption is that each $L_\lambda$ is Fredholm, and so we have a family $L:X\rightarrow\Phi_S(E)$. Denoting as in Appendix \ref{appendix:complexification} by $L^\mathbb{C}_\lambda$, $\lambda\in X$, the complexified operators, we obtain a family $L^\mathbb{C}:X\rightarrow\Phi_S(E^\mathbb{C})$ of selfadjoint Fredholm operators on the complex Hilbert space $E^\mathbb{C}$.

\begin{theorem}\label{mainbif}
Let $X$ be a compact, connected and orientable topological manifold of dimension $k\geq 2$ and assume that

\begin{itemize}
	\item[(i)] there is $\lambda_0\in X$ such that $L_{\lambda_0}$ is invertible,
	\item[(ii)] $c_1(\sind L^\mathbb{C})\neq 0\in H^1(X;\mathbb{Z})$.
\end{itemize} 
Then the covering dimension of the set $B$ of all bifurcation points in $X$ is at least $k-1$, and $B$ is not contractible to a point. 
\end{theorem}

\begin{proof}
We divide the proof into two steps as the argument is different depending on whether or not $X\setminus B$ is connected.\\
\vspace{0.1cm}\\
\textbf{Step 1:} We assume that $X\setminus B$ is not connected. Accordingly, the reduced homology group $\widetilde{H}_0(X\setminus B;\mathbb{Z})$ is non-trivial. As $X$ is connected, the long exact sequence of homology gives

\begin{align*}
\ldots\rightarrow H_1(X,X\setminus B;\mathbb{Z})\rightarrow\widetilde{H}_0(X\setminus B;\mathbb{Z})\rightarrow\widetilde{H}_0(X;\mathbb{Z})=0
\end{align*}
and hence yields a surjective map $H_1(X,X\setminus B;\mathbb{Z})\rightarrow\widetilde{H}_0(X\setminus B;\mathbb{Z})$ implying the non-triviality of $H_1(X,X\setminus B;\mathbb{Z})$. Finally, using that $B$ is compact, we obtain by Poincar\'e-Lefschetz duality (cf. \cite[Cor. VI.8.4]{Bredon}) an isomorphism

\begin{align*}
H_1(X,X\setminus B;\mathbb{Z})\xrightarrow{\cong} \check{H}^{k-1}(B;\mathbb{Z})
\end{align*}
showing that $\check{H}^{k-1}(B;\mathbb{Z})\neq 0$, where $\check{H}$ denotes \v{C}ech-cohomology as in Appendix \ref{appendix-c}. Hence $B$ is not contractible to a point and, moreover, we obtain $\dim B\geq k-1$ by Lemma \ref{cech}.\\
Let us note that in this part of the proof, we do not even need the orientability of $X$ as the same argument could also be done with $\mathbb{Z}_2$ coefficients (cf. \cite{NilsBif}).
\vspace{0.2cm}\\
\textbf{Step 2:} We now assume that $X\setminus B$ is connected. By the universal coefficient theorem, the duality pairing 

\[\langle\cdot,\cdot\rangle:H^1(X;\mathbb{Z})\times H_1(X;\mathbb{Z})\rightarrow\mathbb{Z}\]
is non-degenerate. Consequently, as $c_1(\sind L^\mathbb{C})\neq 0\in H^1(X;\mathbb{Z})$, there exists $\xi\in H_1(X;\mathbb{Z})$ such that

\begin{align}\label{nontriv}
\langle c_1(\sind L^\mathbb{C}),\xi\rangle\neq 0\in\mathbb{Z}.
\end{align}  
Let $\eta\in\check{H}^{k-1}(X;\mathbb{Z})$ denote the Poincar\'e dual of $\xi$, where we use that $X$ is compact and orientable. According to \cite[Cor. VI.8.4]{Bredon}, we have a commutative diagram

\begin{align}\label{exact}
\begin{split}
\xymatrix{&\check{H}^{k-1}(X;\mathbb{Z})\ar[r]^(.47){i^\ast}&\check{H}^{k-1}(B;\mathbb{Z})\\
H_{1}(X\setminus B;\mathbb{Z})\ar[r]^(.57){j_\ast}&H_{1}(X;\mathbb{Z})\ar[u]\ar[r]^(.37){\pi_\ast}&H_{1}(X,X\setminus B;\mathbb{Z})\ar[u]
}
\end{split}
\end{align}
where the vertical arrows are isomorphisms given by Poincar\'e-Lefschetz duality and the lower horizontal sequence is part of the long exact homology sequence of the pair $(X,X\setminus B)$. Because of the commutativity, the class $i^\ast\eta$ is dual to $\pi_\ast\xi$ and we now assume by contradiction the triviality of the latter one.\\
By the exactness of the lower horizontal sequence, there is some $\beta\in H_1(X\setminus B;\mathbb{Z})$ such that $\xi=j_{\ast}\beta$. We now consider the fundamental group $\pi_1(X\setminus B,\lambda_0)$, where $\lambda_0$ denotes the parameter value from the assertion of the theorem for which $L_{\lambda_0}$ is invertible. As $X\setminus B$ is connected, the Hurewicz homomorphism

\[\Gamma:\pi_1(X\setminus B,\lambda_0)\ni[\gamma]\mapsto \gamma_\ast[S^1]\in H_1(X\setminus B;\mathbb{Z}) \]
is surjective, where $[S^1]$ denotes the standard generator of $H_1(S^1;\mathbb{Z})$ (cf. \cite[IV,\S 3]{Bredon}). We choose $[\gamma]\in\pi_1(X\setminus B,\lambda_0)$ such that $\Gamma([\gamma])=\beta$ and set $\overline{\gamma}:=j\circ\gamma\circ q:(I,\partial I)\rightarrow (X,\lambda_0)$, where $q$ is the quotient map $q:(I,\partial I)\rightarrow (S^1,1)$ obtained by collapsing $\partial I$ to a point and $I$ denotes the unit interval. Note that the path $\overline{\gamma}$ does not intersect the bifurcation set $B$.\\
Let us now consider the one-parameter family of functionals

\[\overline{f}:I\times E\rightarrow\mathbb{R},\quad \overline{f}(t, u)=f(\overline{\gamma}(t),u).\] 
Clearly, $\overline{f}$ is continuous, each $\overline{f}_t:=\overline{f}(t,\cdot):H\rightarrow\mathbb{R}$ is $C^2$ and

\[d_u\overline{f}_t=d_uf_{\overline{\gamma}(t)},\quad d^2_u\overline{f}_t=d^2_uf_{\overline{\gamma}(t)},\quad t\in I.\]
Consequently, the critical points of $\overline{f}$ are precisely the critical points of $f$ that lie on the trace of $\overline{\gamma}$. Also every bifurcation point of $\overline{f}$ is a bifurcation point of $f$ (but not vice versa), and the Riesz representation of $d^2_0\overline{f}_t$, which we denote by $\overline{L}_t$, is just $L_{\overline{\gamma}(t)}$.\\
By the main theorem of \cite{BifJac} there is a bifurcation point of critical points of $\overline{f}$ if $\overline{L}_0, \overline{L}_1$ are invertible and $\sfl(\overline{L},I)\neq 0$. Note that $\overline{L}_0, \overline{L}_1$ are indeed invertible, but as $\overline{\gamma}$ does not intersect the bifurcation set $B$, the spectral flow of $\overline{L}$ needs to vanish. Using Lemma \ref{Booss} and \eqref{nontriv}, we get

\begin{align*}
0&=\sfl(\overline{L},I)=\sfl(\overline{L}^\mathbb{C},I)=\sfl(L^\mathbb{C}\circ\overline{\gamma},I)=\sfl(L^\mathbb{C}\circ j\circ\gamma\circ q,I)=\langle c_1(\sind(L^\mathbb{C}\circ j\circ\gamma)),[S^1]\rangle\\
&=\langle c_1(\sind(\gamma^\ast j^\ast L^\mathbb{C})),[S^1]\rangle=\langle \gamma^\ast j^\ast c_1(\sind( L^\mathbb{C})),[S^1]\rangle=\langle c_1(\sind L^\mathbb{C}),j_\ast\gamma_\ast[S_1]\rangle\\
&=\langle c_1(\sind L^\mathbb{C}),j_\ast\beta\rangle=\langle c_1(\sind L^\mathbb{C}),\xi\rangle\neq 0
\end{align*}
which is a contradiction.\\
Hence, $\pi^\ast\xi\neq 0\in H_1(X,X\setminus B;\mathbb{Z})$ and so $\iota^\ast\eta\neq 0\in \check{H}^{k-1}(B;\mathbb{Z})$ showing the non-triviality of the latter group. Consequently, as $k\geq 2$, $B$ is non-contractible and its dimension is greater or equal to $k-1$ by Lemma \ref{cech}.
\end{proof}

\begin{rem}
Note that Theorem \ref{mainbif} improves the main theorem of \cite{Ale} in two ways: Firstly, as already mentioned in the introduction, we no longer assume that $X$ is a smooth manifold and $f:X\times E\rightarrow\mathbb{R}$ is $C^2$. Instead, we just require a continuous dependence on the parameter and a topological manifold as parameter space. Secondly, it is shown in \cite{Ale} that under the assumptions (i) and (ii) the bifurcation set $B$ either disconnects $X$ or it is not contractible to a point. Here we point out that $B$ is never contractible to a point following an idea from \cite{NilsBif}.
\end{rem}
\noindent
In \cite{AleBifIch} the authors proved a multiparameter bifurcation theorem by considering the spectral flow and paths in the parameter space. The following lemma shows that also Theorem \ref{mainbif} can be reformulated as an assertion about the spectral flow of paths, which is important for the proof of Theorem \ref{example} below.

\begin{lemma}\label{lemma-invsfl}
Let $X$ be a compact path-connected space and $L:X\rightarrow\Phi_S(E)$ a family of selfadjoint Fredholm operators. Then the assumptions (i) and (ii) in Theorem \ref{mainbif} hold if and only if there exists a path $\gamma:I\rightarrow X$ such that $\gamma(0)=\gamma(1)$, $L_{\gamma(0)}$ is invertible and $\sfl(L\circ\gamma,I)\neq 0$. 
\end{lemma} 

\begin{proof}
Let the assumptions (i) and (ii) hold and let us set $c:=c_1(\sind L^\mathbb{C})\in H^1(X;\mathbb{Z})$. It is a general fact from algebraic topology that the non-triviality of $c$ entails the existence of a closed path $\gamma:S^1\rightarrow X$ such that $\gamma^\ast c\neq 0\in H^1(S^1;\mathbb{Z})$. Indeed, otherwise $0=\langle\gamma^\ast c,[S^1]\rangle=\langle c,\gamma_\ast[S^1]\rangle$ for all $\gamma:S^1\rightarrow X$, where $[S^1]$ denotes as before the standard generator of $H_1(S^1;\mathbb{Z})$ and $\langle\cdot,\cdot\rangle:H^1(X;\mathbb{Z})\times H_1(X;\mathbb{Z})\rightarrow\mathbb{Z}$ is the duality pairing. Since $\gamma_\ast[S^1]\in H_1(X;\mathbb{Z})$ is the image of $\gamma\in\pi_1(X)$ under the surjective Hurewicz homomorphism $\pi_1(X)\rightarrow H_1(X;\mathbb{Z})$, we infer that $\langle c,x\rangle=0$ for all $x\in H_1(X;\mathbb{Z})$. This shows that $c=0$ as the duality pairing is non-degenerate.\\
Now let $\tilde{\gamma}:S^1\rightarrow X$ be a path that pulls back $c$ to a non trivial element and let $\lambda_0$ be as in (i). Since $X$ is connected, there is a path $\gamma'$ that connects $\lambda_0$ and some point of $\tilde{\gamma}(S^1)$. We define a new closed path $\hat{\gamma}:S^1\rightarrow X$ by running at first through $\gamma'$, then passing $\tilde{\gamma}$ and finally running back to $\lambda_0$ through $\gamma'$ in inverse direction. Then $\hat{\gamma}\simeq\tilde{\gamma}:S^1\rightarrow X$ are homotopic and we obtain by using the functoriality of $c_1$ and $\sind$

\begin{align*}
0\neq\tilde{\gamma}^\ast\,c=\hat{\gamma}^\ast\,c=c_1(\hat{\gamma}^\ast\sind( L^\mathbb{C}))=c_1(\sind(L^\mathbb{C}\circ\hat{\gamma}))\in H^1(S^1;\mathbb{Z}).
\end{align*}
We set $\gamma:=\hat{\gamma}\circ q$, where $q:(I,\partial I)\rightarrow (S^1,1)$ is the canonical quotient map and get by Lemma \ref{Booss} and \eqref{realcomplex} that

\[\sfl(L\circ\gamma,I)=\sfl(L^\mathbb{C}\circ\gamma,I)=\langle c_1(\sind(L^\mathbb{C}\circ\hat{\gamma})),[S^1]\rangle\neq 0.\]
Conversely, let us assume that a closed path $\gamma:I\rightarrow X$ as in the assertion is given, and let $\overline{\gamma}:S^1\rightarrow X$ be a path such that $\overline{\gamma}\circ q=\gamma$. Then (i) holds trivially, and by using again \eqref{realcomplex} and Lemma \ref{Booss} we get that 

\begin{align*}
0&\neq\sfl(L\circ\gamma,I)=\sfl(L^\mathbb{C}\circ\gamma,I)=\langle c_1(\sind(L^\mathbb{C}\circ\overline{\gamma})),[S^1]\rangle\\
&=\langle\overline{\gamma}^\ast\,c_1(\sind L^\mathbb{C}),[S^1]\rangle=\langle c_1(\sind L^\mathbb{C}),\overline{\gamma}_\ast[S^1]\rangle.
\end{align*} 
Consequently, $c_1(\sind L^\mathbb{C})\neq 0\in H^1(X;\mathbb{Z})$, which is (ii). 
\end{proof}
\noindent
In \cite{Ale} the first named author studied multiparameter bifurcation for families of geodesics in semi-Riemannian manifolds which are parametrised by a smooth compact oriented manifold $X$. Let us briefly explain why Theorem \ref{mainbif} does not apply to this setting. By introducing local coordinates, the problem of bifurcation of geodesics can be reduced to bifurcation of critical points of a family of functionals $f:X\times E\rightarrow\mathbb{R}$ as above, where $E$ is the common Sobolev space $H^1_0(I,\mathbb{R}^k)$. The corresponding operators $L_\lambda$ are of the form

\[\langle L_\lambda u,v\rangle_E=\int^1_0{\langle\mathcal{J}u'(t),v'(t)\rangle dt}-\int^1_0{\langle S_\lambda(t)u(t),v(t)\rangle dt},\quad u,v\in E,\]
where $\mathcal{J}=\diag(1,\ldots,1,-1,\ldots,-1)$ and $S_\lambda(t)$ is a smooth family of symmetric matrices (cf. \cite{MussoPejsachowicz}). These operators are selfadjoint and Fredholm, and all basic assumptions of Theorem \ref{mainbif} are satisfied. However, if we write $L_\lambda=T+K_\lambda$, where 

\[\langle T u,v\rangle_E=\int^1_0{\langle\mathcal{J}u'(t),v'(t)\rangle dt},\quad u,v\in E,\]
then it is readily seen that $T\in\Phi_S(E)$ is invertible. Moreover, as $\langle K_\lambda u,v\rangle_E$, $u,v\in E$, extends to a bounded bilinear form on $L^2(I,\mathbb{R}^k)$ and the embedding $H^1_0(I,\mathbb{R}^k)\hookrightarrow L^2(I,\mathbb{R}^k)$ is compact by the Rellich compactness theorem, it is also easily seen that $K_\lambda$ is compact for all $\lambda\in X$. Hence, we obtain by Lemma \ref{indtriv} that $\sind(L^\mathbb{C})=0\in K^{-1}(X)$ and so Theorem \ref{mainbif} cannot provide any information about possible bifurcation points. However, let us point out that the multiparameter bifurcation problem for geodesics in semi-Riemannian manifolds was investigated later by the authors in \cite{AleBifIch}.


\section{Homoclinics of Hamiltonian Systems}
In this section we deal with homoclinic solutions of families of Hamiltonian systems \eqref{Hamiltoniannonlin}. We introduce at first a family of functionals having as critical points the solutions of \eqref{Hamiltoniannonlin} and subsequently we show that the bifurcation invariant in Theorem \ref{mainbif} can indeed be non-trivial in this setting.

\subsection{The Variational Setting}
For a compact topological space $X$ let $\mathcal{H}:X\times\mathbb{R}\times\mathbb{R}^{2n}\rightarrow\mathbb{R}$ be a continuous map such that $\mathcal{H}_\lambda:=\mathcal{H}(\lambda,\cdot,\cdot):\mathbb{R}\times\mathbb{R}^{2n}\rightarrow\mathbb{R}$ is $C^2$ for all $\lambda\in X$ and its derivatives depend continuously on $\lambda\in X$. As in the introduction, we consider the family of Hamiltonian systems

\begin{equation}\label{HamiltoniannonlinII}
\left\{
\begin{aligned}
Ju'(t)+\nabla_u \mathcal{H}_\lambda(t,u(t))&=0,\quad t\in\mathbb{R}\\
\lim_{t\rightarrow\pm\infty}u(t)&=0.
\end{aligned}
\right.
\end{equation}
In what follows, we assume that

\begin{align}\label{Hgrowth}
\mathcal{H}_\lambda(t,u)=\frac{1}{2}\langle A(\lambda,t)u,u\rangle+G(\lambda,t,u),
\end{align}
where $A:X\times\mathbb{R}\rightarrow\mathcal{L}(\mathbb{R}^{2n})$ is a family of symmetric matrices, $G(\lambda,t,u)$ vanishes up to second order at $u=0$, and there are $p>0$, $C\geq 0$ and $g\in H^1(\mathbb{R},\mathbb{R})$ such that

\begin{align}\label{Ggrowth}
|D^2_uG(\lambda,t,u)|\leq g(t)+C|u|^p,\quad u\in\mathbb{R}^{2n},\, t\in\mathbb{R}.
\end{align}
Moreover, we assume that $A_\lambda:=A(\lambda,\cdot):\mathbb{R}\rightarrow\mathcal{L}(\mathbb{R}^{2n})$ converges uniformly in $\lambda$ to families 

\[A_\lambda(+\infty):=\lim_{t\rightarrow\infty}A_\lambda(t),\quad A_\lambda(-\infty):=\lim_{t\rightarrow-\infty}A_\lambda(t),\quad\lambda\in X,\]
and the matrices $JA_\lambda(\pm\infty)$ are hyperbolic, i.e. they have no eigenvalues on the imaginary axis.\\
Note that by \eqref{Hgrowth}, $\nabla_u \mathcal{H}_\lambda(t,0)=0$ for all $(\lambda,t)\in I\times\mathbb{R}$, and so $u\equiv 0$ is a solution of \eqref{HamiltoniannonlinII} for all $\lambda\in X$.\\
Let $C^1_0(\mathbb{R},\mathbb{R}^{2n})$ be the Banach space of all continuously differentiable $\mathbb{R}^{2n}$-valued functions $u$ such that $u$ and $u'$ vanish at infinity, where the norm is defined by

\[\|u\|=\sup_{t\in\mathbb{R}}|u(t)|+\sup_{t\in\mathbb{R}}|u'(t)|.\]

\begin{defi}\label{defibif}
We call $\lambda^\ast\in X$ a bifurcation point for homoclinic solutions from the stationary branch if every neighbourhood of $(\lambda^\ast,0)\in X\times C^1_0(\mathbb{R},\mathbb{R}^{2n})$ contains a non-trivial solution $(\lambda,u)$ of \eqref{HamiltoniannonlinII}.
\end{defi}
\noindent
Let us now briefly recall the variational formulation of the equations \eqref{HamiltoniannonlinII} from \cite[\S 4]{Jacobo}. The bilinear forms $b(u,v)=\langle J u',v\rangle_{L^2(\mathbb{R},\mathbb{R}^{2n})}$, $u,v\in H^1(\mathbb{R},\mathbb{R}^{2n})$, extend to bounded forms on the well known fractional Sobolev space $H^\frac{1}{2}(\mathbb{R},\mathbb{R}^{2n})$, which can be described in terms of Fourier transforms (cf. eg. \cite[\S 10]{Stuart}). Under the assumption \eqref{Hgrowth}, 

\[f_\lambda:H^\frac{1}{2}(\mathbb{R},\mathbb{R}^{2n})\rightarrow\mathbb{R},\quad f_\lambda(u)=b(u,u)+\int^\infty_{-\infty}{\langle A(\lambda,t)u(t),u(t)\rangle\,dt}+\int^\infty_{-\infty}{G(\lambda,t,u(t))\,dt}\]
are $C^2$-functionals such that $f:X\times H^\frac{1}{2}(\mathbb{R},\mathbb{R}^{2n})\rightarrow\mathbb{R}$ is continuous and all its derivatives are continuous as in \eqref{contder}. A careful examination of $f$ shows that the critical points of $f_\lambda$ belong to $C^1_0(\mathbb{R},\mathbb{R}^{2n})$ and are the classical solutions of the differential equation \eqref{HamiltoniannonlinII}. Moreover, every bifurcation point of critical points of $f$ is a bifurcation point of \eqref{HamiltoniannonlinII} in the sense of Definition \ref{defibif}. Finally,
the second derivative of $f_\lambda$ at the critical point $0\in H^\frac{1}{2}(\mathbb{R},\mathbb{R}^{2n})$ is given by 

\[d^2_0f_\lambda(u,v)=b(u,v)+\int^\infty_{-\infty}{\langle A(\lambda,t)u(t),v(t)\rangle\,dt}\]
and, by using the hyperbolicity of $J A_\lambda(\pm\infty)$, it can be shown that the corresponding Riesz representations $L_\lambda:H^\frac{1}{2}(\mathbb{R},\mathbb{R}^{2n})\rightarrow H^\frac{1}{2}(\mathbb{R},\mathbb{R}^{2n})$ are Fredholm. It follows by elliptic regularity that elements in the kernel of $L_\lambda$ are precisely the solutions of the linear differential equation

\begin{equation}\label{Hamiltonianlin}
\left\{
\begin{aligned}
Ju'(t)+A(\lambda,t)u(t)&=0,\quad t\in\mathbb{R}\\
\lim_{t\rightarrow\pm\infty}u(t)&=0.
\end{aligned}
\right.
\end{equation}
Consequently, we obtain from Theorem \ref{mainbif}:

\begin{theorem}\label{mainHam}
Let $X$ be a compact orientable topological manifold of dimension $k\geq 2$, and assume that

\begin{itemize}
	\item[(i)] there is $\lambda_0\in X$ such that \eqref{Hamiltonianlin} has no non-trivial solutions,
	\item[(ii)] $c_1(\sind L^\mathbb{C})\neq 0\in H^1(X;\mathbb{Z})$.
\end{itemize} 
Then the Lebesgue covering dimension of the set $B$ of all bifurcation points of \eqref{HamiltoniannonlinII} in $X$ is at least $k-1$, and $B$ is not contractible to a point. 
\end{theorem}
\noindent
Of course, it is now important to ensure that the assumptions of Theorem \ref{mainHam} can indeed occur, which is the aim of the following section. 

\subsection{A Non-Trivial Example}
We consider for $k\in\mathbb{N}$, $k\geq 2$, the $k$-torus $T^k=S^1\times\cdots\times S^1$ and we identify points $\lambda=(\lambda_1,\ldots,\lambda_k)\in T^k$ with elements $(\Theta_1,\ldots,\Theta_k)\in[-\pi,\pi]^k$ by $\lambda_j=e^{i\Theta_j}$, $1\leq j\leq k$. Note that, as $H^1(T^k;\mathbb{Z})\cong\mathbb{Z}^k$, this is a suitable space for finding non-trivial examples of our bifurcation invariant.\\
In what follows, we consider families of Hamiltonian systems \eqref{HamiltoniannonlinII} for $n=1$ which are parametrised by $X=T^k$. Let $G:T^k\times\mathbb{R}\times\mathbb{R}^{2}\rightarrow\mathbb{R}$ be any map as in \eqref{Hgrowth} satisfying \eqref{Ggrowth}, and

\begin{align*}
A(\lambda,t)=\begin{cases}
(\arctan t)JS_{\Theta_1+\ldots+\Theta_k},\quad &t\geq 0\\
(\arctan t) JS_0,\quad &t<0,
\end{cases}
\end{align*}
where

\[S_{\Theta}=\begin{pmatrix}
\cos\Theta&\sin\Theta\\
\sin\Theta&-\cos\Theta
\end{pmatrix}.\]
Note that the matrices $A(\lambda,t)$ are indeed symmetric, and

\[\lim_{t\rightarrow+\infty}JA(\lambda,t)=-\frac{\pi}{2}S_{\Theta_1+\ldots+\Theta_k},\quad \lim_{t\rightarrow-\infty}JA(\lambda,t)=\frac{\pi}{2}S_{0} \]
are hyperbolic, where the convergence is uniform in $\lambda=(\lambda_1,\ldots,\lambda_k)$.\\
As before, we let $B\subset T^k$ be the set of all bifurcation points as in Definition \ref{defibif}, and we now claim:

\begin{theorem}\label{example}
The covering dimension of $B$ is at least $k-1$, and $B$ is not contractible to a point. 
\end{theorem}

\subsubsection*{Proof of Theorem \ref{example}}
By Theorem \ref{mainHam}, we have to show that there is $\lambda_0\in T^k$ such that $L_{\lambda_0}$ is invertible and that $c_1(\sind L^\mathbb{C})\neq 0$. We divide the proof into three steps:\\
\vspace{0.2cm}\\
\textbf{Step 1:} By Lemma \ref{lemma-invsfl} we only need to find a closed path $\gamma:[a,b]\rightarrow T^k$ such that $L_{\gamma(t)}$ is invertible for some $t\in [a,b]$ and $\sfl(L\circ\gamma,[a,b])\neq 0$. For this we set $\Theta_2=\ldots=\Theta_k=0$ so that our parameter space is $S^1$. Hence we have to compute the spectral flow of the path $\{\widetilde{L}_{\Theta_1}\}_{\Theta_1\in [-\pi,\pi]}$ given by 

\begin{align}\label{tildeL}
\langle\widetilde{L}_{\Theta_1}u,v\rangle_{H^\frac{1}{2}(S^1,\mathbb{R}^2)}=b(u,v)+\int^{\infty}_{-\infty}{\langle\widetilde{A}(\Theta_1,t)u(t),v(t)\rangle\, dt},
\end{align}
where

\begin{align*}
\widetilde{A}(\Theta_1,t)=\begin{cases}
(\arctan t)JS_{\Theta_1},\quad &t\geq 0\\
(\arctan t) JS_0,\quad &t<0,
\end{cases}
\end{align*}
and $\Theta_1\in[-\pi,\pi]$.\\
\vspace{0.1cm}\\
\textbf{Step 2:} By elliptic regularity, the operator $\widetilde{L}_{\Theta_1}$ is non-invertible if and only if the differential equation

\begin{equation}\label{Hamiltonianlinproof}
\left\{
\begin{aligned}
Ju'(t)+\widetilde{A}(\Theta_1,t)u(t)&=0,\quad t\in\mathbb{R}\\
\lim_{t\rightarrow\pm\infty}u(t)&=0,
\end{aligned}
\right.
\end{equation}
has a non-trivial solution.\\ 
Let us recall that the stable and unstable spaces of \eqref{Hamiltonianlinproof} are

\begin{align*}
E^s(\Theta_1,0)&=\{u(0)\in\mathbb{R}^2:\,Ju'(t)+\widetilde{A}(\Theta_1,t)u(t)=0,\, t\in\mathbb{R};\, u(t)\rightarrow 0, t\rightarrow\infty\}\\ 
E^u(\Theta_1,0)&=\{u(0)\in\mathbb{R}^2:\,Ju'(t)+\widetilde{A}(\Theta_1,t)u(t)=0,\, t\in\mathbb{R};\, u(t)\rightarrow 0, t\rightarrow-\infty\}.
\end{align*}
The space $\mathbb{R}^2$ is symplectic with respect to the canonical symplectic form $\omega_0(u,v)=\langle Ju,v\rangle_{\mathbb{R}^{2}}$. As $J\widetilde{A}(\Theta_1,t)$ converges uniformly in $\Theta_1$ to families of hyperbolic matrices for $t\rightarrow\pm\infty$, it can be shown that $E^s(\Theta_1,0)$ and $E^u(\Theta_1,0)$ are Lagrangian subspaces of $\mathbb{R}^2$ (cf. e.g. \cite[Lemma 4.1]{Homoclinics}). This implies in particular that $E^s(\Theta_1,0)$ and $E^u(\Theta_1,0)$ are one-dimensional for all $\Theta_1\in[-\pi,\pi]$.\\
Clearly, there is a non-trivial solution of \eqref{Hamiltonianlinproof} if and only if $E^u(\Theta_1,0)\cap E^s(\Theta_1,0)\neq \{0\}$. By a direct computation, one verifies that

\begin{align*}
u_-(t)&=\sqrt{t^2+1}e^{-t\arctan(t)}\begin{pmatrix} 1\\0\end{pmatrix},\, t\leq 0,\\ u_+(t)&=\sqrt{t^2+1}e^{-t\arctan(t)}\begin{pmatrix} \cos\left(\frac{\Theta_1}{2}\right)\\ \sin\left(\frac{\Theta_1}{2}\right)\end{pmatrix},\,t\geq 0,
\end{align*}
are solutions of \eqref{Hamiltonianlinproof} on the negative and positive half-line, respectively. As they extend to global solutions and since $t\arctan(t)\rightarrow \infty$ as $t\rightarrow\pm\infty$, we see that $u_-(0)\in E^u(\Theta_1,0)$ and $u_+(0)\in E^s(\Theta_1,0)$. As $u_+(0)$ and $u_-(0)$ are linearly dependent if and only if $\Theta_1=0$, we conclude that \eqref{Hamiltonianlinproof} has a non-trivial solution if and only if $\Theta_1=0$, which is given by

\[u_{\ast}(t)=\sqrt{t^2+1}e^{-t\arctan(t)}\begin{pmatrix} 1\\0\end{pmatrix},\quad t\in\mathbb{R}.\] 
Note that we have shown in particular that \eqref{Hamiltonianlinproof} has no non-trivial solution for $\Theta_1\neq 0$ and so the systems \eqref{HamiltoniannonlinII} satisfy Assumption (i) in Theorem \ref{mainHam}.\\
\vspace{0.1cm}\\
\textbf{Step 3:} In this final step of our proof we need to make at first a digression on a method for computing the spectral flow of paths, that was introduced by Robbin and Salamon in \cite{Robbin-Salamon} and Fitzpatrick, Pejsachowicz and Recht in \cite{SFLPejsachowicz}, respectively.\\
As in Section \ref{section:functionals}, let $L:I\rightarrow\Phi_S(E)$ be a path of selfadjoint Fredholm operators on a real separable Hilbert space. In contrast to Section \ref{section-sfl}, we denote here the parameter of $L$ by $\lambda$ to avoid confusion with the time variable $t$ of our Hamiltonian systems. We assume throughout that $L_0,L_1$ are invertible and that $L$ is differentiable with respect to the parameter $\lambda$. An instant $\lambda_0\in I$ is called a \textit{crossing} of $L$ if $L_{\lambda_0}$ is non-invertible, or equivalently $\ker L_{\lambda_0}\neq\{0\}$. The \textit{crossing form} is the quadratic form on the finite dimensional space $\ker L_{\lambda_0}$ defined by

\[\Gamma(L,\lambda_0):\ker L_{\lambda_0}\rightarrow\mathbb{R},\quad \Gamma(L,\lambda_0)[u]=\langle\dot{L}_{\lambda_0} u,u\rangle_E,\]
where $\dot L_{\lambda_0}$ denotes the derivative with respect to the parameter at $\lambda_0$. A crossing $\lambda_0$ is called \textit{regular} if $\Gamma(L,\lambda_0)$ is non-degenerate, and it can be shown that regular crossings are isolated. The following theorem can be found in \cite[Thm. 4.1]{SFLPejsachowicz}.

\begin{theorem}\label{crossings}
If $L:I\rightarrow\Phi_S(H)$ has invertible endpoints and only regular crossings, then

\[\sfl(L,I)=\sum_{\lambda\in I}{\sgn(\Gamma(L,\lambda))},\]
where $\sgn$ denotes the signature of the quadratic form $\Gamma(L,\lambda)$.
\end{theorem}
\noindent
Note that, as regular crossings are isolated, $\Gamma(L,\lambda)=0$ for all but finitely many $\lambda\in I$ and so it is sensible to sum up all $\sgn(\Gamma(L,\lambda))$.\\
Let us now come back to the Hamiltonian systems \eqref{HamiltoniannonlinII} and the operators $\widetilde{L}_{\Theta_1}$ defined by \eqref{tildeL}. According to Step 2, $\Theta_1=0$ is the only crossing of $\widetilde{L}$ and $\ker(\widetilde{L}_0)$ is spanned by $u_\ast$. Hence we need to consider

\[\Gamma(\widetilde{L},0)[u_\ast]=\int^\infty_{-\infty}{\left\langle\dot{\widetilde{A}}(0,t)u_\ast(t),u_{\ast}(t)\right\rangle\,dt},\]
where

\begin{align*}
\dot{\widetilde{A}}(0,t)=\begin{cases}
(\arctan t)J\dot{S}_{0},&\quad t\geq 0\\
0,&\quad t<0,
\end{cases}
\end{align*}
and

\[\dot{S}_{0}=\begin{pmatrix}
0&1\\
1&0	
\end{pmatrix}.\]
We obtain

\begin{align*}
\Gamma(\widetilde{L},0)[u_\ast]&=\int^\infty_{0}{\left\langle\dot{\widetilde{A}}(0,t)u_\ast(t),u_{\ast}(t)\right\rangle\,dt}+\int^0_{-\infty}{\left\langle\dot{\widetilde{A}}(0,t)u_\ast(t),u_{\ast}(t)\right\rangle\,dt}\\
&=\int^\infty_{0}{\arctan(t)\langle J\dot{S}_0 u_{\ast}(t),u_{\ast}(t)\rangle\,dt}\\
&=-\int^\infty_{0}{\arctan(t)(t^2+1)e^{-2t\arctan(t)} \,dt}<0,
\end{align*}
which shows that $\Gamma(L,0)$ is non-degenerate and of signature $-1$. Consequently, by Theorem \ref{crossings}, $\sfl(\widetilde{L},I)=-1$ and so by Lemma \ref{lemma-invsfl} the assumptions of Theorem \ref{mainHam} are satisfied.


\appendix

\section{$K$-Theory and Chern Classes}\label{appendix-bundles}
The first aim of this section is to give a brief recapitulation of $K$-theory. For more details, we refer the reader to \cite{KTheoryAtiyah}.\\
Let $E$ and $F$ be complex vector bundles over a compact base space $X$ and let us denote for a finite dimensional complex linear space $V$ by $\Theta(V)=X\times V$ the product bundle. The set $\Vect(X)$ of all isomorphism classes of complex vector bundles over $X$ is an abelian monoid with respect to the direct sum $\oplus$, and consequently its Grothendieck completion is an abelian group that we denote by $K(X)$. As a matter of fact, elements in $K(X)$ can be written as formal differences $[E]-[F]$, where $[E],[F]$ denote the isomorphism classes in $\Vect(X)$, and $[E_1]-[F_1]=[E_2]-[F_2]$ if and only if there are $n,m\in\mathbb{N}\cup\{0\}$ such that 

\[E_1\oplus F_2\oplus\Theta(\mathbb{C}^n)\simeq E_2\oplus F_1\oplus\Theta(\mathbb{C}^m),\]
i.e. if these vector bundles are isomorphic. The sum operation in $K(X)$ is

\[([E_1]-[F_1])+([E_2]-[F_2])=[E_1\oplus F_2]-[E_2\oplus F_1],\]
the neutral element is $[E]-[E]$ for any bundle $E$ over $X$, and the inverse of $[E]-[F]$ is just $[F]-[E]$.\\
If $Y$ is another compact space and $f:Y\rightarrow X$ is a continuous map, then we obtain a homomorphism in $K$-theory by

\[f^\ast:K(X)\rightarrow K(Y),\quad [E]-[F]\mapsto [f^\ast E]-[f^\ast F],\]
where $f^\ast E$ and $f^\ast F$ denote the pullback bundles over $Y$. In particular, if $\lambda_0\in X$ is a base point, then we have a canonical inclusion $\iota:\{\lambda_0\}\hookrightarrow X$ which induces a homomorphism $\iota^\ast:K(X)\rightarrow K(\{\lambda_0\})\cong\mathbb{Z}$. The kernel of $K(X)$ is denoted by $\widetilde{K}(X)$ and called the \textit{reduced} $K$-theory group of $X$. Finally, we denote by $\Sigma X$ the suspension of $X$ and set $K^{-1}(X):=\widetilde{K}(\Sigma X)$. By extending maps on $X$ to $\Sigma X$ canonically, every continuous map $f:Y\rightarrow X$ between compact spaces induces a homomorphism $f^\ast:K^{-1}(X)\rightarrow K^{-1}(Y)$. Let us mention for the sake of completeness that there are $K$-theory groups $K^n(X)$ for all $n\in\mathbb{Z}$ and maps between them which make $K$-theory an extraordinary cohomology theory. However, as we do not need it in this paper, we do not go into the details.\\
The \textit{Chern classes} (cf. \cite{MiSta}) are a sequence of maps

\[c_k:\Vect(X)\rightarrow H^{2k}(X;\mathbb{Z}),\quad k\in\mathbb{N},\]
which assign to each complex vector bundle over $X$ even-dimensional cohomology classes such that

\begin{itemize}
	\item $c_k(f^\ast E)=f^\ast c_k(E)$ for $f:Y\rightarrow X$ and $Y$ compact,
	\item $c(E\oplus F)=c(E)\cup c(F)$ for $c:=1+c_1+c_2+\ldots\in H^\ast(X;\mathbb{Z})$,
	\item $c_i(E)=0$ if $i>\dim E$,
	\item if $X$ is homotopy equivalent to a CW-complex, then
	\[c_1:\Vect^1(X)\rightarrow H^2(X;\mathbb{Z})\]
	is a bijection, where $\Vect^1(X)=\{E\in\Vect(X):\,\dim E=1\}$.
\end{itemize}   
The second property in particular implies that $c_1(E\oplus F)=c_1(E)+c_1(F)$ which readily shows that $c_1$ descends to a group homomorphism $c_1:K(X)\rightarrow H^2(X;\mathbb{Z})$. As $H^{k+1}(\Sigma X;\mathbb{Z})\cong H^k(X;\mathbb{Z})$ for all $k\in\mathbb{N}$, the first Chern class can be extended to $K^{-1}(X)$ by

\[c_1:K^{-1}(X)=\widetilde{K}(\Sigma X)\xrightarrow{c_1}H^2(\Sigma X;\mathbb{Z})\cong H^1(X;\mathbb{Z}).\] 
In particular, we obtain a homomorphism $c_1:K^{-1}(S^1)\rightarrow H^1(S^1;\mathbb{Z})\cong\mathbb{Z}$ which can be shown to be bijective. Moreover, for every continuous map $f:Y\rightarrow X$ the diagram

\begin{align*}
\xymatrix{
K^{-1}(X)\ar[r]^{f^\ast}\ar[d]_{c_1}&K^{-1}(Y)\ar[d]^{c_1}\\
H^1(X;\mathbb{Z})\ar[r]^{f^\ast}&H^1(Y;\mathbb{Z})
}
\end{align*}
commutes, i.e. the first Chern class in odd $K$-theory is natural.


\section{Complexification of Real Hilbert Spaces}\label{appendix:complexification}

Let $E$ be a real Hilbert space with scalar product $\langle\cdot,\cdot\rangle_E$ and let us denote by $\mathcal{L}(E)$ the Banach space of all bounded linear operators on $E$. We define $E^\mathbb{C}:=\{u+iv:\, u,v\in E\}$ and

\begin{align*}
(x+iy)(u+iv)&:=(xu-yv)+i(yu+xv),\quad x+iy\in\mathbb{C}, u+iv\in E^\mathbb{C},\\
(u_1+iv_1)+(u_2+iv_2)&:=(u_1+u_2)+i(v_1+v_2),\quad u_1+iv_1,u_2+iv_2\in E^\mathbb{C},\\
\langle u_1+iv_1,u_2+iv_2\rangle_{E^\mathbb{C}}&:=\langle u_1,u_2\rangle_{E}+\langle v_1,v_2\rangle_{E}\\
&-i\langle u_1,v_2\rangle_{E}+i\langle v_1,u_2\rangle_{E},\quad u_1+iv_1,u_2+iv_2\in E^\mathbb{C},
\end{align*}
which turns $E^\mathbb{C}$ into a complex Hilbert space.\\
If $T\in\mathcal{L}(E)$, we obtain a bounded linear operator $T^\mathbb{C}\in\mathcal{L}(E^\mathbb{C})$ by

\[T^\mathbb{C}(u+iv):=Tu+iTv.\] 
It is readily seen from the definitions that the adjoints of $T$ and $T^\mathbb{C}$ satisfy $(T^\mathbb{C})^\ast=(T^\ast)^\mathbb{C}$, so that in particular $T$ is selfadjoint if and only if $T^\mathbb{C}$ is selfadjoint. As

\begin{align}\label{kercomp}
\ker(T^\mathbb{C})=(\ker T)^\mathbb{C},\quad \im(T^\mathbb{C})=(\im T)^\mathbb{C}
\end{align} 
we also obtain immediately that $T\in\Phi(E)$ if and only if $T^\mathbb{C}\in\Phi(E^\mathbb{C})$. By the previous fact, the same holds for $\Phi_S(E)$ and $\Phi_S(E^\mathbb{C})$. Finally, the first equality in \eqref{kercomp} shows that $\lambda\in\mathbb{R}$ is an eigenvalue of $T$ if and only if it is an eigenvalue of $T^\mathbb{C}$. Moreover, the corresponding eigenspaces satisfy

\begin{align}\label{complexeigenspace}
E(T,\lambda)^\mathbb{C}=E(T^\mathbb{C},\lambda)
\end{align}
so that in particular the complex dimension of $E(T^\mathbb{C},\lambda)$ and the real dimension of $E(T,\lambda)$ coincide.


\section{The Covering Dimension}\label{appendix-c}
The aim of this final section of the appendix is to recall briefly the definition of the covering dimension of topological spaces.

\begin{defi}
The (covering) dimension $\dim X$ of a topological space $X$ is the minimal value of $n\in\mathbb{N}\cup\{0\}$ such that every finite open cover of $X$ has a finite open refinement in which no point is included in more than $n+1$ elements. $X$ is called infinite dimensional if no such $n$ exists.
\end{defi}  
\noindent
We denote by $\check{H}^{n}(X,A;G)$ the \v{C}ech-cohomology groups of the pair $(X,A)$ with respect to the abelian coefficient group $G$. 

\begin{defi}
The \textit{cohomological dimension} $\dim_GX$ of the topological space $X$ with respect to an abelian coefficient group $G$ is the largest integer $n$ such that there exists a closed subset $A\subset X$ with $\check{H}^n(X,A;G)\neq 0$. If there is no such number, we set $\dim_GX=\infty$. 
\end{defi}
\noindent
We refer for the following lemma to \cite[VIII.4.A]{Hurewicz}.

\begin{lemma}
If $X$ is compact and Hausdorff, then

\begin{align*}
\dim_GX\leq\dim_{\mathbb{Z}}X\leq\dim X
\end{align*}
for any abelian group $G$.
\end{lemma}
\noindent
As an immediate consequence, we obtain the following result:

\begin{lemma}\label{cech}
Let $X$ be compact and Hausdorff. If $\check{H}^{k}(X;G)\neq 0$ for some abelian coefficient group $G$, then $\dim X\geq k$.  
\end{lemma}


\vspace{1cm}
Alessandro Portaluri\\
Department of Agriculture, Forest and Food Sciences\\
Universit\`a degli studi di Torino\\
Largo Paolo Braccini, 2\\
10095 Grugliasco (TO)\\
Italy\\
E-mail: alessandro.portaluri@unito.it

\vspace{1cm}
Nils Waterstraat\\
School of Mathematics,\\
Statistics \& Actuarial Science\\
University of Kent\\
Canterbury\\
Kent CT2 7NF\\
UNITED KINGDOM\\
E-mail: n.waterstraat@kent.ac.uk

\end{document}